\documentclass[reqno]{article} 

\usepackage{latexsym}
\usepackage{verbatim}
\usepackage{epsfig}
\usepackage{rotating}
\usepackage{amssymb}
\usepackage[T1]{fontenc}
\usepackage{afterpage}
\usepackage{color}
\usepackage{pstricks}

\usepackage{amssymb,amsfonts,amsmath,stmaryrd,amsthm}

\usepackage{lineno,hyperref}
\modulolinenumbers[5]

\newtheorem{thm}{Theorem}[section]
\newtheorem{lem}[thm]{Lemma}
\newtheorem{prop}[thm]{Proposition}

\newtheorem{cor}[thm]{Corollary}
\newtheorem{conj}[thm]{Conjecture}

\newtheorem{defn}[thm]{Definition}
\theoremstyle{remark}
\newtheorem{rem}[thm]{Remark}
\newtheorem{ex}[thm]{Example}

\newcommand{\vertical}[0]{\text{Vert}}
\newcommand{\horiz}[0]{\text{Horiz}}
\newcommand{\inn}[0]{\text{In}}
\newcommand{\out}[0]{\text{Out}}
\newcommand{\bin}[0]{\text{Bin}}

\newcommand{\shuffles}[0]{\Sigma}

\newcommand{\hpath}[0]{\mathcal{H}}
\newcommand{\vpath}[0]{\mathcal{V}}
\newcommand{\NN}{\mathbb{N}_0}

\begin{document}

\title{Signed Enumeration of Upper-Right Corners in Path Shuffles}

\author{William Kuszmaul}

\maketitle

\begin{abstract}
We resolve a conjecture of Albert and Bousquet-M{\'e}lou enumerating quarter-planar walks with fixed horizontal and vertical
projections according to their upper-right-corner count modulo 2. In
doing this, we introduce a signed upper-right-corner count
statistic. We find its distribution over planar walks with any choice
of fixed horizontal and vertical projections. Additionally, we prove
that the polynomial counting loops with a fixed horizontal and
vertical projection according to the absolute value of their signed
upper-right-corner count is $(x+1)$-positive. Finally, we conjecture
an equivalence between $(x+1)$-positivity of the generating function
for upper-right-corner count and signed upper-right-corner count,
leading to a reformulation of a conjecture of Albert and
Bousquet-M{\'e}lou on which their asymptotic analysis of permutations
sortable by two stacks in parallel relies.
\end{abstract}


\section{Introduction}

When studying walks on $\mathbb{Z}^2$ with unit steps, it is natural
to restrict oneself to walks which project vertically on a fixed
vertical path $\vpath$ and horizontally on a fixed horizontal path
$\hpath$, where $\vpath$ (resp. $\hpath$) comprises North and South
steps (resp. East and West steps) \cite{albert15}. Such paths
correspond with \emph{shuffles}, or interleavings, of $\vpath$ and
$\hpath$. If $|\vpath| = v$ and $|\hpath| = h$, there are
${v+h} \choose {h}$ such shuffles. It will be convenient to denote East, West, North,
and South steps by $\rightarrow$, $\leftarrow$, $\uparrow$, and
$\downarrow$.

Given a path, several natural statistics arise. In this paper, we
focus on variants of \emph{peak-count}, which counts the number of NW
(i.e., $\Lsh$) and ES (i.e., \rotatebox[origin=c]{270}{$\Rsh$})
corners. If the path lies within the first quadrant, then the
peak-count is visually the number of peaks pointing away from the
origin. Equivalently, peak-count counts occurrences of $\rightarrow
\downarrow$ and $\uparrow \leftarrow$ in a shuffle of $\vpath$ and
$\hpath$. Albert and Bousquet-M{\'e}lou introduced peak-count and
showed it to have applications in the study of permutations sortable
by two stacks in parallel \cite{albert15}. They also proved that
peak-count has the same distribution over shuffles as do several other
statistics (Proposition 14 of \cite{albert15}, which was originally
observed by Julien Courtiel and Olivier Bernardi independently).

Albert and Bousquet-M{\'e}lou studied several generating functions
involving peak-count of planar walks in depth. Although very little is
known in the case where the path's horizontal and vertical projections
are fixed, Albert and Bousquet-M{\'e}lou posed the following
conjecture, which they attribute to Julien Courtiel \cite{albert15}.

\begin{conj}[Albert and Bousquet-M{\'e}lou, Conjecture (P1) on pp. 32 \cite{albert15}]\label{conjproven}
Let $\hpath$ (resp. $\vpath$) be a path beginning and ending at the
origin of half-length $i$ (resp. $j$) on the alphabet $\{\rightarrow,
\leftarrow\}$ (resp. $\{\uparrow, \downarrow\}$). Then the the
polynomial that counts walks of the shuffle class of $\hpath \vpath$ according to
the number of $\Lsh$ and \rotatebox[origin=c]{270}{$\Rsh$} corners
takes the value $i+j \choose i$ when evaluated at $-1$. Equivalently,
the shuffle class of $\hpath \vpath$ contains $i + j \choose i$ more shuffles
with even peak-count than with odd peak-count.
\end{conj}

Albert and Bousquet-M{\'e}lou note that, among other things, proving the
conjecture would eliminate the need for the somewhat lengthy proof of
their Proposition 15 \cite{albert15}.

We prove Conjecture \ref{conjproven} and extend it to the case where
$\hpath$ and $\vpath$ are arbitrary words on the alphabets
$\{\rightarrow, \leftarrow\}$ and $\{\uparrow, \downarrow\}$
respectively. This corresponds with considering arbitrary walks on the
plane, rather than only planar loops.

In order to study peak-count modulo 2, it is sufficient to study what
we call \emph{signed peak-count}, which is the number of $\Lsh$
corners minus the number of $\rotatebox[origin=c]{270}{$\Rsh$}$
corners. While signed peak-count and peak-count are guaranteed to
share the same parity, it turns out that signed peak-count exhibits
nice behavior allowing for it to be completely enumerated.

\begin{thm}\label{introthm}
Let $\vpath$ be a word comprising $u$ $\uparrow$'s and $d$
$\downarrow$'s. Let $\hpath$ be a word comprising $r$ $\rightarrow$'s
and $l$ $\leftarrow$'s. Then the number of shuffles of $\vpath$ and
$\hpath$ with signed peak-count $k$ is
$${r+u \choose u-k}{l+d \choose d+k}.$$
\end{thm}

Note that Theorem \ref{introthm} immediately leads to a formula for
the difference between the number of even and odd peak-count shuffles of
$\vpath$ and $\hpath$.
\begin{equation}\label{formdiff}
\sum_k (-1)^k{r+u \choose u-k}{l+d \choose d+k}. 
\end{equation}

In certain key cases, Formula \eqref{formdiff} simplifies. When $r =
l$ and $u = d$ (as is the case in Conjecture \ref{conjproven}),
Formula \eqref{formdiff} becomes
$$\sum_k (-1)^k{r+u \choose u-k}{r+u \choose u+k} = {r+u
  \choose u},$$
by equation (30) of \cite{gessel92}.

Additionally, when $r = u$ and $l = d$, Formula
\eqref{formdiff} becomes
$$\sum_k (-1)^k{2r \choose r - k}{2l \choose l + k},$$
which is known \cite[Equation 29]{gessel92} to be the super Catalan number,
$$S(r, k) = \frac{(2r)!(2k)!}{r!k!(r+k)!}.$$ Finding a combinatorial
interpretation for the super Catalan number is an open problem; when
$r \le 3$ or $|r-k| \le 4$, combinatorial interpretations have been
found (\cite{allen14}, \cite{chen12}). Note that our result does not
yield a combinatorial interpretation because it examines the
difference between the cardinalities of two sets, rather than the
cardinality of a single set. As far as we know, however, it is the
first non-contrived example of a combinatorial problem in which the
super Catalan numbers appear.

Peak-count is intimately connected to the study of permutations
sortable by two stacks in parallel. Albert and Bousquet-M{\'e}lou use
the statistic to both characterize the generating function enumerating
such permutations and to analyze the asymptotic number of them
\cite{albert15}. The asymptotic analysis, in particular, brings
peak-count to the spotlight, as the analysis relies on two conjectures
about peak-count which remain open. In addition to the results already
discussed, this paper makes progress on the first of these conjectures
(Conjecture 10 of \cite{albert15}), which concerns the
$(x+1)$-positivity of the generating function for peak-count of
quarter-planar loops. In particular, our results reduce the conjecture
to a new conjecture of a different flavor concerning the relationship
between peak-count and signed peak-count.

The study of permutations sortable by two stacks in parallel traces
back to Section 2.2.1 of \emph{The Art of Computer Programming}
\cite{knuth}, in which Knuth characterized the permutations sortable
by a single stack, as well as those sortable by a double-ended
queue. This spawned the works of Even and Itai \cite{even71}, Pratt
\cite{pratt73}, and Tarjan \cite{tarjan72} which extended Knuth's
investigations to more general networks of stacks and queues,
including the permutations sortable by two stacks in parallel. The
1971 paper of Even and Itai \cite{even71} characterized the
permutations sortable by two stacks in parallel, allowing for the
detection of such a permutation in polynomial time. The 1973 paper of
Pratt \cite{pratt73} then characterized the list of minimal
permutations which cannot be sorted by two stacks in parallel,
connecting the problem to the rapidly growing field of permutation
pattern avoidance\footnote{Interestingly, the origin of the study of
  permutation pattern avoidance is often also traced back to Section
  2.2.1 of \emph{The Art of Computer Programming} \cite{knuth}. The
  first serious study of the combinatorial structures of
  pattern-avoiding permutations, however, was conducted by Simion and
  Schmidt in 1985 \cite{simion85}.}. Despite this, however, the
enumeration\footnote{Several notions of an \emph{enumeration}
  exist. Here we mean a counting technique which can be performed in
  polynomial time. This excludes, for example, the 2012 algorithm of
  Denton \cite{Denton12} which counts the permutations in time
  $O(n^52^n)$.}  of the permutations sortable by two stacks in
parallel remained open for more than forty years, until finally being
solved by Albert and Bousquet-M{\'e}lou in 2014 along with the
introduction of peak-count.

As mentioned previously, Albert and Bousquet-M{\'e}lou's asymptotic
analysis of the number of sortable permutations relies on two
conjectures concerning peak-count \cite{albert15}. We make progress on
the first:

\begin{conj}[Albert, Bousquet-M{\'e}lou, Conjecture 10 on pp. 12 \cite{albert15}]\label{conjunproven}
The generating function as a function of a variable $x$ for peak-count
of quarter-planar loops of a given length can be expressed as a
positive polynomial in $\mathbb{Z}[x+1]$ (i.e., is $(x+1)$-positive).
\end{conj}

In Section \ref{secpos}, we prove a variant of Conjecture
\ref{conjunproven} applying to signed peak-count rather than to
peak-count. In particular, we prove that the polynomial counting loops
according to the absolute value of their signed peak-count is
$(x+1)$-positive, even when we restrict ourselves to loops with fixed
horizontal and vertical projections. In Section \ref{secconj}, we go
on to pose Conjecture \ref{conjmain} which reduces the study of
$(x+1)$-positivity of peak-count to the study of $(x+1)$-positivity of
the absolute value of signed peak-count. Due to our results on
$(x+1)$-positivity of absolute signed peak-count, it follows that
Conjecture \ref{conjmain} implies Conjecture
\ref{conjunproven}.

The outline of this paper is as follows. In Section
\ref{secsignedlcount}, we enumerate signed peak-count of shuffles;
additionally, we enumerate peak-count modulo 2. In Section \ref{secpos},
we prove our results on $(x+1)$-positivity. In Section \ref{secconj},
we conclude with conjectures and directions of future work.

\section{Enumerating Signed peak-count}\label{secsignedlcount}

In this section, we derive a formula for the number of shuffles with a
given signed peak-count. To accomplish this, we introduce a subtly
related statistic called shifted In-Vert and, under certain
conditions, we provide a bijection establishing the equidistribution
of shifted In-Vert and signed peak-count. We begin by introducing
some conventions.

For the rest of this section, let $r, l, u, d$ be non-negative
integers. Moreover, fix $\vpath$ (resp. $\hpath$) to be a word
comprising $r$ $\rightarrow$ steps (resp. $u$ $\uparrow$ steps) and
$l$ $\leftarrow$ steps (resp. $d$ $\downarrow$ steps). In particular,
$\rightarrow$, $\leftarrow$, $\uparrow$, and $\downarrow$ correspond
to East, West, North, and South respectively.

\begin{defn}
A shuffle of $\vpath$ and $\hpath$ is an intertwining of the steps in
$\vpath$ and $\hpath$. In addition, a shuffle has the $\swarrow$ in
its zero-th position.
\end{defn}

\begin{rem}
The convention of $\swarrow$ preceding the first step is
non-standard. But it will facilitate the definition of a statistic
called shifted In-Vert.
\end{rem}

We will use $\shuffles$ to denote the set of shuffles of $\vpath$ and $\hpath$.

\begin{ex}
  If $\vpath = \uparrow \downarrow$ and $\hpath = \rightarrow  \rightarrow$, then $\shuffles$ is the set:
\[
  \begin{array}{r c c c c c l}
\{&\swarrow& \rightarrow& \rightarrow& \uparrow& \downarrow &,\\
&\swarrow& \rightarrow& \uparrow& \rightarrow& \downarrow &,\\
&\swarrow& \uparrow& \rightarrow& \rightarrow& \downarrow &,\\
&\swarrow& \rightarrow& \uparrow& \downarrow& \rightarrow &,\\
&\swarrow& \uparrow& \rightarrow& \downarrow& \rightarrow &,\\
&\swarrow& \uparrow& \downarrow& \rightarrow& \rightarrow &\}\\

  \end{array}
\]
\end{ex}

\begin{defn}
For $\sigma \in \shuffles$, we use $\#(A, B)_\sigma$ to denote the
number of occurrences in $\sigma$ of an element of set $A$ followed
immediately by an element of set $B$. If $A$ or $B$ contains only a
single element, we omit set braces around that element. For example,
$A = \{\leftarrow\}$ is written as $A = \leftarrow$.
\end{defn}

\begin{defn}
We use the following short-hands for certain sets: $\horiz =
\{\rightarrow, \leftarrow\}$ is the set of horizontal steps,
$\vertical = \{\uparrow, \downarrow\}$ is the set of vertical steps,
$\inn = \{\leftarrow, \downarrow, \swarrow\}$ is the set of inward
steps, and $\out = \{\rightarrow, \uparrow \}$ is the set of outward
steps.
\end{defn}

The names \emph{inward} and \emph{outward} take the perspective of the
first quadrant of the plane, in which $\leftarrow$, $\downarrow$ and
$\swarrow$ all point inward towards the origin, and both $\rightarrow$ and
$\uparrow$ point outward away from the origin.

\begin{ex}
The value $\#(\inn, \uparrow)_\sigma$ counts instances of a horizontal step
followed by an $\uparrow$. For example, if $\sigma = \overline{\swarrow,
  \uparrow}, \leftarrow, \downarrow, \rightarrow,
\overline{\downarrow, \uparrow}, \uparrow$, then $\#(\inn, \uparrow)_\sigma
= 2$ and the instances are over-lined.
\end{ex}

\begin{defn}
The signed peak-count of $\sigma \in \shuffles$ is $\#(\uparrow, \leftarrow)_\sigma - \#(\rightarrow, \downarrow)_\sigma$.
\end{defn}

We begin by reformulating signed peak-count:
\begin{lem}\label{lemL}
The signed peak-count of $\sigma \in \shuffles$ can be expressed as $$\#(\uparrow, \inn)_\sigma - \#(\out, \downarrow)_\sigma.$$
\end{lem}
\begin{proof}
By definition, signed peak-count is  $\#(\uparrow, \leftarrow)_\sigma - \#(\rightarrow, \downarrow)_\sigma$. This is equal to 
$$\#(\uparrow, \leftarrow)_\sigma + \#(\uparrow, \downarrow)_\sigma - \#(\rightarrow, \downarrow)_\sigma - \#(\uparrow, \downarrow)_\sigma,$$
which combines to
$$\#(\uparrow, \inn)_\sigma - \#(\out, \downarrow)_\sigma.$$
\end{proof}

Next we introduce \emph{shifted In-Vert}, which is easier to study
than peak-count. Interestingly, Proposition \ref{propbiject} will show
that the two statistics are equidistributed in certain key cases.
\begin{defn}
The \emph{In-Vert} of a shuffle $\sigma \in \shuffles$ is $\#(\inn, \vertical)_\sigma$. The
\emph{shifted In-Vert} is $\#(\inn, \vertical)_\sigma - d$. (Recall that $d$
is the number of $\downarrow$'s in $\sigma$.)
\end{defn}

It will be useful to have the following reformulation of shifted In-Vert.
\begin{lem}\label{lemN}
The shifted In-Vert of $\sigma\in \shuffles$ can be expressed as $$\#(\inn, \uparrow)_\sigma -
\#(\out, \downarrow)_\sigma.$$
\end{lem}
\begin{proof}
Because shuffles cannot begin with $\downarrow$ (due to the
$\swarrow$), every $\downarrow$ is preceded by either an inward or an
outward step. Consequently, we can rewrite $d$ as $(\#(\out,
\downarrow)_\sigma + \#(\inn, \downarrow)_\sigma)$. Thus the shifted
In-Vert $\#(\inn, \vertical)_\sigma - d$ expands to
$$(\#(\inn, \uparrow)_\sigma + \#(\inn, \downarrow)_\sigma) - (\#(\out, \downarrow)_\sigma +
\#(\inn, \downarrow)_\sigma),$$ 
which simplifies to 
$$\#(\inn, \uparrow)_\sigma - \#(\out, \downarrow)_\sigma.$$
\end{proof}

Next we construct an involution $f: \shuffles \rightarrow \shuffles$ such that, under
certain conditions, the shifted In-Vert of $f(\sigma)$ is the same as the
signed peak-count of $\sigma$.

\begin{defn}
A \emph{out-run} in $\sigma \in \shuffles$ is a (non-empty) maximal contiguous subsequence of outward steps.
\end{defn}

\begin{defn}
The \emph{flip} $f(\sigma)$ of a shuffle $\sigma \in \shuffles$, is obtained by replacing every out-run in $\sigma$ with the same out-run written backwards.
\end{defn}

\begin{ex}
For example, if 
$$\sigma = \swarrow, \rightarrow, \rightarrow, \uparrow, \leftarrow, \downarrow, \uparrow, \rightarrow, \leftarrow, \leftarrow, \uparrow,$$ then the out-runs are
$\rightarrow,\rightarrow,\uparrow$, in addition to $\uparrow, \rightarrow$, and $\uparrow$. Thus the flip of $\sigma$ is
$$f(\sigma) = \swarrow, \uparrow, \rightarrow, \rightarrow, \leftarrow, \downarrow, \rightarrow, \uparrow, \leftarrow, \leftarrow, \uparrow.$$
\end{ex}

\begin{prop}\label{propbiject}
Suppose the final step of $\vpath$ is $\downarrow$. Then for all $\sigma \in \shuffles$ the
shifted In-Vert of $f(\sigma)$ is the same as the signed peak-count of $\sigma$.
\end{prop}
\begin{proof}
Since the final step of $\vpath$ is $\downarrow$, if $\sigma$ ends
with a out-run, then that out-run must consist only of
$\rightarrow$'s. Thus every out-run in $\sigma$ containing at least
one $\uparrow$ is both preceded and followed by an inward
step. Consequently, $\#(\inn, \uparrow)_{f(\sigma)} = \#(\uparrow,
\inn)_\sigma,$ since flipping a shuffle reverses the order of every
out-run. Additionally, it is easy to see that $\#(\out,
\downarrow)_{f(\sigma)} = \#(\out, \downarrow)_{\sigma}$.

By Lemma \ref{lemN}, the shifted In-Vert of $f(\sigma)$ is $\#(\inn,
\uparrow)_{f(\sigma)} - \#(\out, \downarrow)_{f(\sigma)},$ which thus equals $\#(\uparrow, \inn)_\sigma -
\#(\out, \downarrow)_\sigma$. By Lemma \ref{lemL}, this is exactly the signed
peak-count of $\sigma$.
\end{proof}

Next, we enumerate In-Vert in the case where the final letter in
$\vpath$ is $\downarrow$.
\begin{prop}\label{propNcount}
Let $\vpath$ and $\hpath$ be vertical and horizontal paths with a
total of $u$ $\uparrow$'s, $d \downarrow$'s, $r$ $\rightarrow$'s, and
$l$ $\leftarrow$'s. Suppose the final step of $\vpath$ is
$\downarrow$. Then the number of shuffles $\sigma \in \shuffles$ with
In-Vert $k$ is
$${r+u \choose u+d-k}{l+d \choose k}.$$
\end{prop}
\begin{proof}
Pick $u+d-k$ of the $\rightarrow$'s and $\uparrow$'s in $\hpath$ and
$\vpath$. Pick $k$ of the $\leftarrow$'s and $\downarrow$'s in
$\hpath$ and $\vpath$. If the final $\downarrow$ in $\vpath$ is
picked, then replace it by picking $\swarrow$ instead. We call the
steps which have just been picked \emph{blue}, and the remaining steps
\emph{red}. We call such a coloring a \emph{blue-red coloring}. By
definition, there are ${r+u \choose u+d-k}{l+d \choose k}$ blue-red
colorings.

Given a blue-red coloring $C$ and a shuffle $\sigma \in \shuffles$, we say that $C$
\emph{fits} $\sigma$ if the blue steps in $C$ are exactly the steps in
$\sigma$ which precede a vertical step. We will show that for every
$\sigma$ with In-Vert $k$ there is a $C$ fitting $\sigma$, and that
for every $C$ there is exactly one $\sigma$ with In-Vert $k$ which $C$
fits. Consequently, the number of blue-red colorings $C$ is equal to
the number of $\sigma \in \shuffles$ with In-Vert $k$.

First we show that for every $\sigma \in \shuffles$ with In-Vert $k$
there is a blue-red coloring $C$ which fits $k$. Select $C$ to color a
step blue if it precedes a vertical step in $\sigma$ and red
otherwise. Because $\sigma$ has In-Vert $k$, exactly $k$ inward steps
will be blue, while $u+d-k$ outward steps will be blue. Additionally,
because the final step of $\vpath$ is $\downarrow$, it cannot precede
a vertical step in $\sigma$, and will not be blue. Thus $C$ is a valid
blue-red coloring which fits $\sigma$.

Next, in order to complete the proof, we show that every blue-red
coloring fits exactly one $\sigma \in \shuffles$ with In-Vert $k$. We
may uniquely construct $\sigma$ as follows. By definition, the first
step is $\swarrow$. Given the first $j$ steps, the $(j+1)$-th step
must be vertical if the $j$-th step is blue, and horizontal
otherwise. Therefore, the $(j+1)$-th step is either forced to be the
next unused step of $\vpath$ or forced to be the next unused step of
$\hpath$. Consequently, all of $\sigma$ is uniquely determined by the
blue-red coloring. Observe that the construction uses exactly $u+d$
steps from $\vpath$ since exactly $u+d$ steps are blue, and thus also
uses exactly $r+l$ steps from $\hpath$. Consequently, the construction
is well-defined, yielding a valid shuffle $\sigma$ with In-Vert $k$.
\end{proof}

Using Proposition \ref{propbiject} and Proposition \ref{propNcount},
it is easy to enumerate signed peak-count when the final step of $\vpath$
is $\downarrow$. In fact, with a little more work, one can remove the
restriction that $\vpath$ ends with $\downarrow$. This brings us to our main
result of the section.
\begin{thm}\label{thmmain}
Let $\vpath$ and $\hpath$ be vertical and horizontal paths with a
total of $u$ $\uparrow$'s, $d \downarrow$'s, $r$ $\rightarrow$'s, and
$l$ $\leftarrow$'s. The number of $\sigma \in \shuffles$ with signed
peak-count $k$ is
$${r+u \choose u-k}{l+d \choose d+k}.$$
\end{thm}
\begin{proof}
There are two cases which Proposition \ref{propbiject} gives us for
free.
  \begin{itemize}
  \item \textbf{Case 1: The final step of $\vpath$ is $\downarrow$.} By Proposition
    \ref{propbiject} the number of shuffles with signed peak-count $k$ is the
    number with shifted In-Vert $k$, and thus In-Vert $d+k$. By
    Proposition \ref{propNcount}, this is $${r+u \choose u+d-(d+k)}{l+d
      \choose (d+k)} = {r+u \choose u-k}{l+d \choose d+k}.$$
  \item \textbf{Case 2: The final step of $\hpath$ is $\leftarrow$.}
    The signed peak-count of a shuffle is the negative of the signed
    peak-count of the complement shuffle in which all $\rightarrow$'s
    become $\uparrow$'s, all $\leftarrow$'s become $\downarrow$'s, and
    vice-versa. By case (1), the number of such complement shuffles
    with signed peak-count $-k$ is
    $${u+r \choose r+k}{d+l \choose l-k} = {r+u \choose u-k}{l+d \choose
      d+k}.$$
  \end{itemize}
Note that if one of $\hpath$ or $\vpath$ is empty, then the theorem is
trivially true. If neither Case 1 nor Case 2 holds, and neither
$\hpath$ nor $\vpath$ are empty paths, then the final steps of
$\vpath$ and of $\hpath$ must both be outward. We will prove this
final case by induction on $r + u$, using the previous cases as base
cases.
  \begin{itemize}
  \item \textbf{Case 3: The final steps of $\vpath$ and of $\hpath$ are both
    outward.} First consider shuffles ending with $\rightarrow$. Ignoring the
    final step, we see by induction that the number
    with signed peak-count $k$ is $${(r-1)+u \choose u-k}{l+d \choose
      d+k}.$$
    
    Next consider shuffles ending with $\uparrow$. Ignoring the final step, we
    see by induction that the number with signed peak-count $k$ is
    $${r+(u-1) \choose (u-1)-k}{l+d \choose d+k}.$$
    
    Summing over the two sub-cases yields $${r+u \choose u-k}{l+d \choose
      d+k}.$$
  \end{itemize}
\end{proof}

Recall that peak-count and signed peak-count share
parity. Consequently, Theorem \ref{thmmain} gives a formula for
peak-count modulo 2:
\begin{cor}\label{cormod2}
The number of $\sigma \in \shuffles$ with even peak-count minus the
number with odd peak-count is
$$\sum_k (-1)^k{r+u \choose u-k}{l+d \choose d+k}.$$ 
\end{cor}

When $r = l$ and $u = d$, this resolves the Conjecture
\ref{conjproven}, since $$\sum_k (-1)^k{r+u \choose u-k}{r+u \choose
  u+k} = {r+u \choose u}.$$ (See Equation (30) of \cite{gessel92}.)
When $a= u$ and $l=d$, the sum yields the super Catalan numbers. (See
Equation (29) of \cite{gessel92}.)

\section{(x+1)-Positivity of Signed Peak-Count of Loops}\label{secpos}

A polynomial in $\mathbb{Z}[x]$ is \emph{$(x+1)$-positive} if it is in
$\NN[x+1]$, where $\NN$ denotes the non-negative integers.

Recall that Conjecture \ref{conjunproven}, which is Conjecture 10 of
\cite{albert15}, states that the polynomial counting quarter-planar
loops of a given length by peak-count is $(x+1)$-positive. In
particular, this result is then used in an analysis of the asymptotic
number of permutations sortable by two stacks in parallel.

In this section, we prove that a similar but stronger statement can be
made about the absolute value of signed peak-count of loops. This is
of particular interest due to conjectures we pose in Section
\ref{secconj} connecting $(x+1)$-positivity of peak-count and signed
peak-count. In fact, our conjectures combined with Theorem
\ref{thmpos} imply Conjecture \ref{conjunproven}.

One approach to proving a polynomial is $(x+1)$-positive is to
evaluate the polynomial at $(x-1)$ and then to prove that the
coefficients of the new polynomial are non-negative. For absolute
value of signed peak-count, however, this approach does not easily
yield results. Instead, we introduce $(x+1)$-positive building-block
polynomials out of which we construct the polynomial counting shuffles
by absolute value of signed peak-count.

We begin by establishing notation.
\begin{defn}
  Let $\vpath$ (resp. $\hpath$) be any path consisting of $u$
  $\uparrow$'s (resp. $r$ $\rightarrow$'s) and $d$ $\downarrow$'s
  (resp. $l$ $\leftarrow$'s). Let $F(r, l, u,d)$ be the generating
  function in the variable $x$ counting shuffles of $\vpath$ and
  $\hpath$ according to the absolute value of the signed peak-count.
\end{defn}

By Theorem \ref{thmmain}, $F(r, l, u, d)$ depends only on $r$, $l$,
$u$, and $d$ and not $\vpath$ or $\hpath$, making it well defined.  The
main result of this section is Theorem \ref{thmpos} which proves that
$F(m, m, n, n)$ is $(x+1)$-positive and results in a formula for $F(m,
m, n, n)$ expanded in powers of $(x+1)$. To accomplish this, we define a
set of words $W^m_n$ and a simple statistic called \emph{absolute even
  count} such that $F(m, m, n, n)$ counts words in $W^m_n$ according
to the statistic. We then partition $W^m_n$ so that the generating
function counting a given part according to absolute even-count is in
a family of $(x+1)$-positive building-block polynomials. We begin by
constructing these building blocks.




\begin{defn}
The polynomial $\bin_k(n) \in \mathbb{Z}[x]$ is defined by $$\bin_k(n)
= \sum_{i \ge 0} x^i {n \choose i+k}.$$
\end{defn}

It is not hard to show that $\bin_k(n)$ is $(x+1)$-positive.
\begin{lem}\label{lemF}
Let $n$ and $k$ be positive integers. Then,
$$\bin_k(n) = \sum_{i \ge 0} {n-i-1 \choose k-1}(x+1)^{i}.$$
\end{lem}
\begin{proof}
  Observe that $\bin_k(n)$ is the generating function counting subsets
  $S \subseteq \{1, \ldots, n\}$ of size at least $k$ according to $|S| - k$.

  Consider only sets $S \subseteq \{1, \ldots, n\}$ such that the $k$-th largest element of $S$
  is $j$. Then there are $j-1 \choose k-1$ options for the $k$
  smallest elements of $S$. Moreover, since we can choose whether each
  of $\{j+1, j+2, \ldots, n\}$ is in $S$, the total contribution of
  such paths to $\bin_k(n)$ is $${j-1 \choose k-1}(x+1)^{n-j}.$$

Summing over values for $j$, we get
$$\bin_k(n) = \sum_{j = 1}^n {j-1 \choose k-1}(x+1)^{n-j} = \sum_{i \ge 0}{n-i-1 \choose k-1}(x+1)^i.$$
\end{proof}

Note that Lemma \ref{lemF} could also have easily been proven by
routinely using Zeilberger's creative telescoping
\cite{zeilberger91}. However, we prefer our combinatorial proof as it
highlights one of the types of arguments that may prove useful to the
reader when studying $(x + 1)$-positivity.

We will now use $\bin_k(n)$ to build a more interesting building block:
\begin{defn}
Let $\bin^k(n)$ be the generating function counting subsets $S \subseteq
[n]$ according to $||S| - k|$.
\end{defn}

The next result, Proposition \ref{propG}, establishes that $\bin^k(n)$
is $(x+1)$-positive. Interestingly, the expansion of $\bin^k(n)$ in
powers of $(x + 1)$ has constant coefficient 0. Like Lemma \ref{lemF},
Proposition \ref{propG} could also be shown with Zeilberger's creative
telescoping. We find it more straightforward, however, to simply
interpret $\bin^k(n)$ in terms of its cousin $\bin_{k'}(n')$.
\begin{prop}\label{propG}
Let $n$ and $k$ be positive integers and $k \le n$. Then,
$$\bin^k(n) = \sum_{i>0} \left({n-i-1 \choose k-1} + {n-i-1 \choose
    n-k-1}\right)(x+1)^{i}.$$
\end{prop}
\begin{proof}
Recall that $\bin_k(n)$ counts subsets $S \subseteq [n]$ of size at least $k$
according to $|S| - k$. On the other hand, $\bin_{n-k}(n)$ counts subsets
$S \subseteq [n]$ of size at least $n-k$ according to $|S| - n +
k$. Equivalently, considering the complement sets, $\bin_{n-k}(n)$ counts
subsets of size at most $k$ according to $(n-|S|) - n + k = k - |S|$. Thus
$$\bin_k(n) + \bin_{n-k}(n) = \bin^k(n) + {n \choose k},$$ since we are
double-counting subsets of size $k$.

It follows from Lemma \ref{lemF} that for $k > 0$, 
$$\bin^k(n) =  \sum_{i \ge 0} {n-i-1 \choose k-1}(x+1)^{i} +  \sum_{i \ge 0} {n-i-1 \choose n-k-1}(x+1)^{i} - {n \choose k}$$
$$= \sum_{i>0} \left({n-i-1 \choose k-1} + {n-i-1 \choose
  n-k-1}\right)(x+1)^{i} + {n-1 \choose k-1} + {n-1 \choose
  n-k-1} - {n \choose k}.$$
Reducing the final three terms to zero, we get the desired formula.
\end{proof}

Having designed $(x + 1)$-positive building-block polynomials into which we will partition $F(m, m, n, n)$, we next introduce the idea of absolute even-count.

\begin{defn}
Let $W^m_n$ be the set of binary words with $2m$ zeroes and $2n$ ones.
\end{defn}

\begin{defn}
The \emph{even-count} of a binary word $w$ is the number of ones in
even-indexed positions. The \emph{shifted even-count} is
$$(\#\text{ ones in even positions}) - \frac{1}{2}(\#\text{ ones}).$$
The \emph{absolute even-count} is the absolute value of the shifted
even-count.
\end{defn}

\begin{ex}
The word $1010110000$ has absolute even-count $|1 - 2|=1$.
\end{ex}

Absolute even-count's usefulness stems from its connection to signed
peak-count, which is described in the following lemma.
\begin{lem}\label{lemevencount}
If $\hpath$ is a horizontal loop with $m$ steps right and left, and
$\vpath$ is a vertical loop with $n$ steps up and down, then number of
shuffles of $\vpath$ and $\hpath$ with signed peak-count of absolute
value $k$ is the same as the number of elements of $W^m_n$ with
absolute even-count $k$.

Consequently, $F(m, m, n, n)$ counts elements of $W^m_n$ by absolute
even-count.
\end{lem}
\begin{proof}
  Theorem \ref{thmmain} tells us that the number of shuffles of
  $\vpath$ and $\hpath$ with signed peak-count of absolute value $k$
  is
  $$2{m+n \choose n-k}{m+n \choose n+k}.$$

  On the other hand, in order for an element of $W^m_n$ to have $n +
  k$ ones in even positions, we get to choose $n + k$ of the $m + n$
  even positions to take value one, and then $n - k$ of the $m + n$
  odd positions to take value one. Similarly, in order for an element
  of $W^m_n$ to have $n - k$ ones in even positions, we get to choose
  $n - k$ of the even positions to be ones, and $n + k$ of the odd
  positions to be ones. It follows that the number of elements in
  $W^m_n$ with absolute even-count $k$ is also
  $$2{m+n \choose n-k}{m+n \choose n+k}.$$
  
\end{proof}

It will be useful for us to think about absolute even-count using as
simple reformulation involving the notion of an odd-indexed pair.

\begin{defn}
An \emph{odd-indexed pair} in a word $w$ is any pair of letters $(w(2i+1), w(2i+2))$.
\end{defn}

\begin{ex}
The binary word $10011100$ partitions into odd-indexed pairs as $10|01|11|00$.
\end{ex}

\begin{lem}\label{lemevencount2}
For a given $w \in W^m_n$, the absolute even-count of $w$ is
$$|(\# \ (0, 1) \text{ odd-indexed pairs in }w) - \frac{1}{2}(\#  \ (a, b) \text{
  odd-indexed pairs in }w \text{ with }a \neq b)|.$$
\end{lem}
\begin{proof}
  For a given $w \in W^m_n$, the contribution of a given odd-indexed
  pair $(a, b)$ to the shifted even-count is 0 if $a = b$;
  $-\frac{1}{2}$ if $a =1$ and $b = 0$; and $\frac{1}{2}$ if $b = 1$
  and $a = 0$. Therefore, the shifted even-count of $w$ is

  $$\frac{1}{2}(\# \ (0, 1) \text{ odd-indexed pairs in }w) -
  \frac{1}{2}(\# \ (1, 0) \text{ odd-indexed pairs in }w),$$
  which can be rewritten as
  $$(\# \ (0, 1) \text{ odd-indexed pairs in }w) - \frac{1}{2}(\#  \ (a, b) \text{
  odd-indexed pairs in }w \text{ with }a \neq b).$$
\end{proof}

We are now prepared to present the main result of the section. Theorem
\ref{thmpos} describes $F(m, m, n, n)$ as a polynomial in
$(x + 1)$.
\begin{thm}\label{thmpos}
  The polynomial $F(m, m, n, n)$ is $(x + 1)$-positive. More precisely, $F(m, m, n, n)$ equals
  $${m+n \choose n} + 2  \sum_{i>0} (x+1)^i \sum_{0 \le k < n} {m+n \choose k}{m+n-k \choose 2n-2k}{2n-2k-i-1 \choose n-k-1}.$$
\end{thm}

\begin{proof}
Lemma \ref{lemevencount} allows us to interpret $F(m, m, n, n)$ as
counting elements of $W^m_n$ according to their absolute
even-count. We will partition the elements of $W^m_n$ so that each
part has absolute even-count enumerated by $\bin^k(2k)$ for some $k$,
allowing us to apply Proposition \ref{propG} in order to prove the
$(x+1)$-positivity of $F(m, m, n, n)$.

For $0 \le i < m+n$, define \emph{$i$-toggling} of a word $w \in W^m_n$
as swapping the values in the $(i + 1)$-th odd-indexed pair. Say that two
words are \emph{toggle-equivalent} if they can be reached from each
other by repeated toggling (for various $i$). Denote the
toggle-equivalence class of $w \in W^m_n$ by $T(w)$. For example, $T(110110)$ is
$$\{11|01|10,\phantom{\}}$$
$$\phantom{\{}11|10|10,\phantom{\}}$$
$$\phantom{\{}11|01|01,\phantom{\}}$$
$$\phantom{\{}11|10|01\},$$
where we use $|$ to separate odd-indexed pairs.

Call the set of $i$ such that $w(2i+1)=1$ and $w(2i+2)=1$ the
\emph{anchor} $A(w)$, and the set of $i$ such that $w(2i) \neq
w(2i+1)$ the \emph{base} $B(w)$. Then we can characterize $T(w)$ in
terms of $w$'s anchor and base, observing that
\begin{equation}\label{eqT}
T(w) = \{w' : w' \in W^m_n, \ A(w') = A(w), \ B(w') = B(w)\}.
\end{equation}

Equation \eqref{eqT} tells us that we can think of each element $w'$
of $T(w)$ as being represented by the set $S_{w'} = \{i \mid w'(2i +
1) = 0, w'(2i + 2) =1\}$, with each of the subsets $S \subseteq B(w)$
representing exactly one element in $T(w)$.  Moreover, by Lemma
\ref{lemevencount2}, the absolute even-count of each element $w'$ is
then given by $||S_{w'}| - |B(w)| / 2|$. It follows that the elements
of $T(w)$ are counted by $\bin^{|B(w)|/ 2}(|B(w)|)$ with respect to
absolute even-count.

At this point, we have shown that $F(m, m, n, n)$ can be expressed as
the sum of polynomials of the form $\bin^{j}(2j)$, each of which we
know to be $(x + 1)$-positive due to Proposition
\ref{propG}\footnote{Proposition \ref{propG} technically doesn't cover
  $\bin^0(0) = 1$, which is obviously also $(x +
  1)$-positive.}. Having established the $(x + 1)$-positivity of $F(m,
m, n, n)$, it remains to derive a formula for it as a polynomial in
$(x + 1)$.

There are ${m+n \choose k}{m+n-k \choose 2n-2k}$ distinct
toggle-equivalence classes containing words with anchors of size
$k$. In particular, there are $m+n \choose k$ options for $A(w)$, for
each of which there are $m+n-k \choose 2n-2k$ options for
$B(w)$. Therefore the distribution of absolute even-count over
$W^m_n$ is given by the polynomial
\begin{equation}\label{eqtobin}
\sum_{0 \le k \le n} {m+n \choose k}{m+n-k \choose 2n-2k}\bin^{n-k}(2n-2k).
\end{equation}
By Proposition \ref{propG}, when $k \neq n$,
$\bin^{n-k}(2n-2k)$ equals
$$\sum_{i>0} \left({(2n-2k)-i-1 \choose (n-k)-1} + {(2n-2k)-i-1 \choose
  (2n-2k)-(n-k)-1}\right)(x+1)^{i}$$
\begin{equation}\label{eqbinsimp}
  = 2\sum_{i>0} {2n-2k-i-1 \choose n-k-1}(x+1)^{i}.
\end{equation}
When $k = n$, $\bin^{n-k}(2n-2k) =1$. Plugging this along with Equation
\eqref{eqbinsimp} into Equation \eqref{eqtobin}, the distribution of
absolute even-count over $W^m_n$ is
$${m+n \choose n}{m+n-n \choose 0}$$ $$+ 2 \sum_{0 \le k < n} {m+n \choose k}{m+n-k \choose 2n-2k} \sum_{i>0} {2n-2k-i-1 \choose n-k-1}(x+1)^{i}$$
$$= {m+n \choose n} + 2  \sum_{i>0} (x+1)^i \sum_{0 \le k < n} {m+n \choose k}{m+n-k \choose 2n-2k}{2n-2k-i-1 \choose n-k-1}.$$
\end{proof}

It turns out that Theorem \ref{thmmain} extends to proving the $(x +
1)$-positivity of $F(m, n, m, n)$ at no additional cost. Indeed,
Proposition \ref{propscale} shows that the $(x+1)$-positivity of $F(r,
l, u, d)$ is always symmetric in $l$ and $u$. In addition, since $F(r,
l, u, d) = F(u, d, r, l)$, it follows that $(x+1)$-positivity is also
symmetric in $d$ and $r$.
\begin{prop}\label{propscale}
  The polynomial $F(r, l, u, d)$ is $(x+1)$-positive if and only if $F(r, u, l, d)$ is $(x+1)$-positive. More precisely,
  $$F(r, l, u, d) \cdot (r + l)! (u + d)! = F(r, u, l, d) \cdot (r + u)!(l + d)!.$$
\end{prop}
\begin{proof}
  Observe that $$F(r, l, u, d) = \sum_i x^{|i|} {r + u \choose u + i}{l + d \choose d - i},$$
  by Theorem \ref{thmmain}. Pulling out a factor of $\frac{(r+u)!(l+d)!}{(r+l)!(u+d)!}$, we obtain
  $$F (r, l, u, d) = \frac{(r+u)!(l+d)!}{(r+l)!(u+d)!} \sum_i x^{|i|} {r+l \choose
    l+i}{u+d \choose d - i}$$
  $$= \frac{(r+u)!(l+d)!}{(r+l)!(u+d)!}F(r, u,
  l, d),$$ completing the proof.  
\end{proof}

\section{Conjectures and Open Questions}\label{secconj}
In this section, we state conjectures relating peak-count and signed
peak-count, and we discuss directions for future work. Recall that
Albert and Bousquet-M{\'e}lou's recent asymptotic analysis of
permutations sortable by two stacks in parallel relies on two
conjectures, the most interesting of which is Conjecture
\ref{conjunproven}. Our
main conjecture, Conjecture \ref{conjmain}, implies Conjecture
\ref{conjunproven}, and may prove easier for future researchers to tackle.

We call a path \emph{vertical} (resp. \emph{horizontal}) if it
comprises $\uparrow$ and $\downarrow$ (resp. $\rightarrow$ and
$\leftarrow$) steps. A path is \emph{quarter-planar} if it resides in
the quarter plane $\{(x, y)| x \ge 0, y \ge 0\}$. A quarter-planar
vertical or horizontal path is referred to as \emph{positive}.

Fix $\vpath$ to be a vertical positive path with $u$ up-steps and $d$
down-steps. Let $Q$ be the set of quarter-planar paths with $r$
right-steps, $l$ left-steps, and vertical projection $\vpath$. Let
$G_1$ be the generating function for peak-count over $Q$ and $G_2$ be
the generating function for signed peak-count over $Q$. We will state
our conjectures in terms of $G_1$ and $G_2$. Additionally,
computations indicate that all of our conjectures are still true when
we modify $Q$ to consider planar paths rather than quarter-planar
paths and $\vpath$ to be any vertical path rather than necessarily a
positive one\footnote{Note that the modified versions of the
  conjectures neither imply or are implied by the un-modified
  versions. The modified versions may be easier though. For example,
  the modified version of Conjecture \ref{conjalb} is solved in
  \cite{albert15}, and the modified version of Conjecture
  \ref{conjx1equal} is straightforward.}.

We begin by re-stating a conjecture of \cite{albert15}, originally
posed as Conjecture (P2). Conjecture
\ref{conjalb} implies several other conjectures of \cite{albert15},
including Conjecture
\ref{conjunproven}, the primary conjecture on which their asymptotic
analysis of two-stack sortable permutations relies.

\begin{conj}[Albert and Bousquet-M{\'e}lou, Conjecture (P2) on pp. 32 \cite{albert15}] \label{conjalb}
  If $a = b$ and $c =d$, then $G_1$ is $(x+1)$-positive.
\end{conj}

Our main conjecture, Conjecture \ref{conjmain}, states a surprising
equivalence between $(x+1)$-positivity of $G_1$ and $G_2$ (tested for
$u,d,r,l \le 5$). In particular, given Conjecture \ref{conjmain}, our
analysis of $(x+1)$-positivity of signed peak-count in Theorem
\ref{thmpos} implies Conjecture \ref{conjalb}.

\begin{conj}\label{conjmain}
  The polynomial $G_1$ is $(x+1)$-positive if and only if $G_2$ is $(x+1)$-positive.
\end{conj}

Since peak-count and signed peak-count are of equal parity, the constant
coefficients of $G_1$ and $G_2$ as polynomials in $(x+1)$ are
equal. Computations show, however, that $(x+1)$-positivity of $G_1$
and $G_2$ does not depend only on this coefficient.

By Theorem \ref{thmmain}, $G_2$ depends only on $r$, $l$, $u$, and
$d$. Thus Conjecture \ref{conjmain} implies the following.

\begin{conj}\label{conjx1equal}
  The $(x+1)$-positivity of $G_1$ depends only on $r$, $l$, $u$, and $d$.
\end{conj}

This conjecture is particularly surprising given that $G_1$ itself
does not depend only on $r$, $l$, $u$, and $d$. Can one characterize for which
$r, l, u, d$ $G_1$ and $G_2$ are $(x+1)$-positive? Proposition
\ref{propscale} shows that $(x+1)$-positivity of $G_2$ is unchanged by
swapping $d$ and $r$ or $u$ and $l$.

In the proof of Theorem \ref{thmpos}, we essentially partitioned
shuffles of vertical and horizontal loops $\vpath$ and $\hpath$ so
that each part has a signed peak-count distribution counted by the
generating function $\bin^{j}(2j)$ for some $j \in \NN$. It is natural
to ask whether this proof technique extends to analyzing $G_2$ for
arbitrary $r, l, u, d$. Observe that $B = \{1\} \cup
\{\frac{\bin^{j}(2j)}{2} : j \in \mathbb{N}\} = \{1\} \cup
\{\bin^{j}(2j-1) | j \in \mathbb{N}\}$ contains a single monic
polynomial of each degree\footnote{Note that $2\bin^j(2j-1) =
  \bin^j(2j)$ for $j \in \mathbb{N}$.}, implying that any polynomial
in $\mathbb{Z}[x]$ can be expressed uniquely as a linear combination
of elements in $B$. If a polynomial can be expressed as linear
combination of elements of $B$ with non-negative
coefficients\footnote{We do not require the coefficient of each
  $\bin^j(2j-1)$ to be even, as is the case when $a = b$ and $c
  =d$. In fact, such a requirement would make Conjecture
  \ref{conjbuild} untrue.}  call it \emph{toggle-buildable}. The
following conjecture has been tested for $r, l, u, d \le 5$.

\begin{conj}\label{conjbuild}
 The generating function $G_2$ is toggle-buildable exactly when $G_2$ is $(x+1)$-positive.
\end{conj}

However, the same cannot be said for peak-count (i.e., $G_1$), even
when $r = l$ and $u = d$ (so that we are considering quarter-planar
loops with a fixed vertical projection). For example, when
$r=l=4$ and $u=d=2$, the polynomial $G_1$, which equals $x^2 + 12x + 15 = (1 +
x)^2 + 10(1 + x) + 4$, is $(x+1)$-positive but not toggle-buildable.

Our proof of Theorem \ref{thmpos} suggests several additional
questions. In the proof of Theorem \ref{thmpos}, we discuss $F(m, m,
n, n)$ in terms of the absolute even-count of binary words. Assuming
the final step of $\vpath$ is $\downarrow$, however, then the proof
of Proposition \ref{propNcount} provides a straightforward bijection
between binary words in $W^m_n$ with absolute even-count $k$ and
shuffles of $\vpath$ and $\hpath$ with shifted In-Vert
$k$. Furthermore, if we apply Proposition \ref{propbiject}, then we
obtain a bijection between binary words in $W^m_n$ with absolute even
count $k$ and shuffles of $\vpath$ and $\hpath$ with signed peak-count
$k$. This leads us to pose several open problems:
\begin{enumerate}
\item When $i$-toggling a binary word changes its even-count by one,
  it also changes the signed peak-count of the corresponding shuffle by
  one. Computer computations indicate that, in fact, when $i$-toggling
  changes signed peak-count by one, it often also only changes peak-count by
  one. If one could characterize when this is the case, then toggling
  could be a useful tool for proving Conjectures \ref{conjalb} and
  \ref{conjmain}.
\item Conjecture \ref{conjbuild} suggests that the the toggling
  argument in the proof of Theorem \ref{thmpos} may have the potential to be
  extended to $F(m, n, m, n)$ (without using the trick from the proof
  of Proposition in which we multiply the expression by a non-integer
  amount before analyzing it). In particular, such an extension might
  yield a combinatorial interpretation for the super Catalan
  number. Indeed, the shuffles in toggle-equivalence classes of size
  one would be exactly those counted by the constant term of $F(m, n,
  m, n)$ expanded in powers of $(x + 1)$, of which there are a super
  Catalan number by Corollary \ref{cormod2}.
\item Suppose $\hpath$ is of the
  form $$\rightarrow, \rightarrow, \rightarrow, \ldots, \rightarrow,
  \leftarrow, \ldots, \leftarrow, \leftarrow, \leftarrow,$$ and $\vpath$ is of the
  form  $$\uparrow,\uparrow,\uparrow,\ldots, \uparrow, \downarrow,
  \ldots, \downarrow, \downarrow, \downarrow,$$ each with
  the same number of outward steps as inward steps. Then any shuffle
  of $\vpath$ and $\hpath$ cannot contain both $(\uparrow,
  \leftarrow)$ and $(\rightarrow, \downarrow)$ pairs. Thus the
  absolute value of the signed peak-count is equal to the peak-count
  of such shuffles. Consequently, Theorem \ref{thmpos} proves the
  $(x+1)$-positivity of the generating function counting these
  shuffles by their peak-count.

  However, we have been unable to find a simple combinatorial
  description for the action of $i$-toggling on these shuffles (i.e.,
  the result of $i$-toggling the corresponding binary word). Does one
  exist? And is there an analogue of toggling for peak-count rather than
  signed peak-count?
\end{enumerate}

\section{Acknowledgments}

This research was conducted at the University of Minnesota Duluth REU
and was supported by NSF grant 1358695 and NSA grant
H98230-13-1-0273. The author thanks Joe Gallian for suggesting the
problem, as well as Timothy Chow, Aaron Pixton, and Adam Hesterberg
for offering advice on exposition and directions of research.

\newpage

\bibliography{writeup}

\end{document}